\documentclass[journal]{IEEEtran}

\usepackage{amsmath,amssymb,enumerate,subfigure,multirow,hhline,color}
\usepackage{xcolor}
\usepackage{hyperref}
\usepackage{graphicx}
\usepackage{graphics}
\usepackage[sort]{cite}
\usepackage{amsthm}
\usepackage{algorithm}
\usepackage{algpseudocode}

\usepackage{tcolorbox}
\usepackage{epsfig,psfrag}
\usepackage{tabularx}
\usepackage{tabulary}

\usepackage[sort]{cite}

\usepackage{hyperref}
\hypersetup{breaklinks=true,colorlinks=true,linkcolor=blue,citecolor=blue}

\usepackage[noabbrev,capitalise,nameinlink]{cleveref}

\DeclareMathOperator{\arginf}{arg\,inf}
\DeclareMathOperator{\argmin}{arg\,min}

\newtheorem{theorem}{Theorem}

\newtheorem{assumption}{Assumption}

\newtheorem{corollary}{Corollary}
\newtheorem{problem}{Problem}
\newtheorem{lemma}{Lemma}
\newtheorem{definition}{Definition}
\newtheorem{proposition}{Proposition}

\allowdisplaybreaks

\title{Convergence of Dynamic Programming on the Semidefinite Cone}

\author{Donghwan Lee
\thanks{D. Lee is with the Department of Electrical Engineering,
KAIST, Daejeon, 34141, South Korea {\tt\small
donghwan@kaist.ac.kr}.}
}

\begin{document}

\maketitle

\begin{abstract}
The goal of this paper is to investigate new and simple convergence analysis of dynamic programming for linear quadratic regulator problem of discrete-time linear time-invariant systems. In particular, bounds on errors are given in terms of both matrix inequalities and matrix norm. Under a mild assumption on the initial parameter, we prove that the Q-value iteration exponentially converges to the optimal solution. Moreover, a global asymptotic convergence is also presented. These results are then extended to the policy iteration. We prove that in contrast to the Q-value iteration, the policy iteration always converges exponentially fast. An example is given to illustrate the results.
\end{abstract}

\section{Introduction}
The optimal control theory for linear systems has a long tradition~\cite{bellman1965dynamic,bertsekas1996neuro,bertsekas2005dynamic,lewis2012optimal}. At the core of the optimal control problem is the dynamic programming: it offers a general and effective paradigm for finding optimal policies for the optimal control problem. Early works of dynamic programming~\cite{caines1970discrete,bertsekas1996neuro,bertsekas2005dynamic} clarified many issues, such as asymptotic convergence and precise conditions to guarantee the convergence. More recently, progresses have been made for nonlinear systems~\cite{heydari2014revisiting,al2008discrete} and switched linear systems~\cite{zhang2009value,zhang2011infinite,lincoln2006relaxing} to name just a few. A natural next step is to understand its non-asymptotic behavior: does the algorithm make a consistent and quantifiable progress toward the optimal solution? Although an exponential convergence has been established in~\cite{zhang2009value} under a stronger assumption on the weighting matrix, the question still remains unsettled.

Motivated by the discussions, in this paper, we revisit the classical results on the convergence analysis of dynamic programming in different angles for discrete-time linear time-invariant systems. In particular, the classical analysis usually focuses on the value function iteration. On the other hand, in this paper, we pay more attentions to the Q-function iteration, which is relevantly more popular in the field of reinforcement learning~\cite{sutton1998reinforcement,lewis2009reinforcement,bradtke1994adaptive}, in particular, Q-learning~\cite{watkins1992q,bradtke1994adaptive}. However, the proposed analysis can be directly applied to the value function-based dynamic programming as well. Most importantly, we study the convergence of Q-value iteration (Q-VI) and Q-function-based policy iteration (Q-PI) in terms of matrix inequality bounds and matrix norm. We prove that the error of Q-VI has an upper bound expressed in terms of matrix inequalities on the semidefinite cone, and the upper bounding matrix converges exponentially fast. On the other hand, its lower bound can be also expressed in terms of matrix inequalities, while the lower bound is proven to converge asymptotically. The overall convergence rate is dominated by the asymptotic behavior of the lower bound. However, it turns out that when the initial parameter lies on a certain semidefinite cone, the error matrix is upper and lower bounded by matrices that exponentially converge to zero. Therefore, under this scenario, an exponential convergence of the error can be derived in terms of some matrix norm. The validity of the results is demonstrated through an example. As a next step, the results for Q-VI are extended to the analysis of Q-PI. In particular, similar analysis can be applied to Q-PI except for one aspect. In contrast to Q-VI, Q-PI always guarantees exponential convergence independently of the initial parameters. This improvement comes from the additional initial information of Q-PI: the stabilizing gain initially given to Q-PI. We expect that the present work sheds new light on more exact analysis of dynamic programming with different angles, which inherit the simplicity and elegance.

{\bf Notation}: The adopted notation is as follows: ${\mathbb R}$: set of real numbers; ${\mathbb R}^n $: $n$-dimensional Euclidean
space; ${\mathbb R}^{n \times m}$: set of all $n \times m$ real
matrices; $A^T$: transpose of matrix $A$; $A^{-T}$: transpose of matrix $A^{-1}$; $A \succ 0$ ($A \prec
0$, $A\succeq 0$, and $A\preceq 0$, respectively): symmetric
positive definite (negative definite, positive semi-definite, and
negative semi-definite, respectively) matrix $A$; $I$: identity matrix with appropriate dimensions; ${\mathbb S} ^n $: symmetric $n \times
n$ matrices; ${\mathbb S}_+^n : = \{ P \in {\mathbb S}^n :P \succeq 0\}$; ${\mathbb S}_{++}^n : = \{ P \in {\mathbb S}^n :P \succ 0\}$; $\rho(\cdot)$: spectral radius; $\lambda _{\max } ( \cdot )$: maximum eigenvalue; $\lambda _{\min} ( \cdot )$: minimum eigenvalue.

\section{Problem Formulation and Preliminaries}\label{section:infinite-horizon-LQR-problem}

Consider the discrete-time linear time-invariant (LTI) system
\begin{align}
&x(k + 1) = Ax(k) + Bu(k),\quad x(0) = z \in {\mathbb
R}^n,\label{eq:LTI-system}
\end{align}
where the integer $k \geq 0$ is the time, $x(k) \in {\mathbb R}^n$ is the state
vector, $u(k) \in {\mathbb R}^m$ is the input vector, and $z \in
{\mathbb R}^n$ is the initial state. Assuming the input, $u(k)$, is given by a state-feedback control
policy, $u(k)=Fx(k)$, we denote by $x(k;F,z)$ the solution of
\eqref{eq:LTI-system} starting from $x(0)=z$. Under the
state-feedback control policy, the cost function for the classical
linear quadratic regulator (LQR) problem is denoted by
\begin{align*}
&J(F,z):= \sum_{k = 0}^\infty{\gamma^k \begin{bmatrix}
   x(k;F,z)\\
   Fx(k;F,z)\\
\end{bmatrix}^T \Lambda \begin{bmatrix}
   x(k;F,z)\\
   Fx(k;F,z)\\
\end{bmatrix}},
\end{align*}
where $\Lambda := \begin{bmatrix}
   Q & 0\\
   0 & R\\
\end{bmatrix}\succeq 0$ is the weight matrix and $\gamma \in (0,1]$ is called the discount factor.
By introducing the augmented state vector $v(k):=
\begin{bmatrix}
  x(k)\\
  u(k)\\
\end{bmatrix}$, we consider the augmented system throughout the paper
\begin{align}
&v(k+1)=A(F) v(k),\quad v(0) = v_0  \in {\mathbb R}^{n +
m},\label{eq:augmented-system}
\end{align}
where
\[
A(F): = \left[ {\begin{array}{*{20}c}
   A & B  \\
   {FA} & {FB}  \\
\end{array}} \right] \in {\mathbb R}^{(n + m) \times (n + m)}.
\]
If $v_0=\begin{bmatrix} z^T & z^T F^T \\
\end{bmatrix}^T$, then the state and input parts of $v(k)$ are
identical to $x(k)$ and $u(k)$ in~\eqref{eq:LTI-system}. A useful
property of $A(F)$ is that its spectral radius, $\rho (A(F))$, is
identical to that of $A+BF$.
\begin{lemma}[\cite{lee2018primal}]\label{lemma:spetral-radius-lemma}
$\rho(A+BF)= \rho (A(F))$ holds.
\end{lemma}

Define ${\cal F}_\gamma$ as the set of all stabilizing state-feedback
gains of system $(\gamma^{1/2}A, \gamma^{1/2}B)$, i.e., ${\cal F}_\gamma:= \{F \in {\mathbb
R}^{m \times n} :\rho(\gamma^{1/2} A+ \gamma^{1/2} BF)<1\}$. Then, ${\cal F}_\gamma$ is
an open set, not necessarily convex~\cite[Lemma~2]{geromel1998static}; however, finding a
state feedback gain $F\in {\cal F}_\gamma$ can be reduced to a simple
convex problem. Notice that with $\gamma = 1$, ${\cal F}_1$ is the set of all stabilizing state-feedback gains of $(A,B)$. From the standard LQR theory, although $J^*(F,z)$ has different
values for different $z \in {\mathbb R}^n$, the minimizer
$F^*=\argmin_{F \in {\mathbb R}^{m\times n}} J(F,z)$ is not dependent on $z$.
Based on this notion, the infinite-horizon LQR problem is formalized below.
\begin{problem}[Infinite-horizon LQR problem]\label{problem:infinite-horizon-LQR}
For any $z\in {\mathbb R}^n$, solve $F^*=\argmin_{F\in {\mathbb R}^{m\times n}} J(F,z)$ if the optimal value of $\inf_{F\in {\mathbb R}^{m\times n}} J(F,z)$ exists and is attained.
\end{problem}

For a given $z\in {\mathbb R}^n$, if the optimal
value of $\inf_{F \in {\mathbb R}^{m\times n}} J(F,z)$ exists and is attained,
then the optimal cost is denoted by $J^*(z)=J(F^*,z)$.
Assumptions that will be used throughout the paper are summarized
below.
\begin{assumption}\label{assumption:basic-assumption}
Throughout the paper, we assume that
\begin{itemize}
\item $Q \succeq 0,R \succ 0$;

\item $(A,B)$ is stabilizable and $Q$ can be written as $Q=C^T
C$, where $(A,C)$ is detectable.
\end{itemize}
\end{assumption}

Under~\cref{assumption:basic-assumption}, the optimal
value of $\inf_{F \in {\mathbb R}^{m\times n}} J(F,\,z)$ exists, is attained, and
$J^*(z)$ is a quadratic function, i.e., $J^*(z) = z^T X^* z$,
where $X^*$ is the unique solution of the algebraic Riccati
equation (ARE)~\cite[Proposition~4.4.1]{bertsekas2005dynamic} for $X$:
\begin{align*}
X =& \gamma A^T XA - \gamma A^T XB(R + \gamma B^T XB)^{-1} \gamma B^T XA + Q,\\
X\succeq& 0.
\end{align*}
In this case, $J^*(z)$ as a function of $z\in {\mathbb R}^n$ is
called the optimal value function, which satisfies the Bellman equation
\[
J^* (z) = \min _{w \in {\mathbb R}^m } \left\{ {\left[ {\begin{array}{*{20}c}
   z  \\
   w  \\
\end{array}} \right]^T \Lambda \left[ {\begin{array}{*{20}c}
   z  \\
   w  \\
\end{array}} \right] + \gamma J^* (Az + Bw)} \right\}
\]

The reader can refer to~\cite{bertsekas2005dynamic}
and~\cite{kwakernaak1972linear} for more details of the classical
LQR results. The corresponding optimal control policy is $u^* (z)
= F^* z$, where
\begin{align}
&F^*:=-(R + \gamma B^T X^* B)^{-1} \gamma B^T X^*A \in {\cal F}_\gamma \label{eq:F*}
\end{align}
is the unique optimal gain. Note that $F^*\in {\cal F}_\gamma$, i.e., it stabilizes $(\gamma^{1/2}A,\gamma^{1/2}B)$. Alternatively, the optimal $Q$-function~\cite{bertsekas2005dynamic} is defined as
\begin{align}
&Q^*(z,u):=z^T Qz + u^T Ru + \gamma J^*(Az+Bu)= \begin{bmatrix}
   z  \\
   u  \\
\end{bmatrix}^T P^* \begin{bmatrix}
   z  \\
   u  \\
\end{bmatrix},\label{eq:Q-factor}
\end{align}
where
\begin{align}
&P^*:=\begin{bmatrix}
   Q + \gamma A^T X^* A & \gamma A^T X^* B \\
   \gamma B^T X^*A & R + \gamma B^T X^*B  \\
\end{bmatrix}.\label{eq:P*}
\end{align}
Once the optimal Q-function is found, then the optimal policy can be expressed as
\begin{align*}
&u^*(z)=F^*z=\argmin_{u \in {\mathbb R}^m}Q^*(z,u).
\end{align*}

The optimal Q-function is known to satisfy the Q-Bellman equation
\[
Q^* (z,u) = \left[ {\begin{array}{*{20}c}
   z  \\
   u  \\
\end{array}} \right]^T \Lambda \left[ {\begin{array}{*{20}c}
   z  \\
   u  \\
\end{array}} \right] + \gamma \min _{w \in {\mathbb R}^m } Q^* (Az + Bu,w)
\]
and its parametric form is
\begin{align*}
P^* =& \left[ {\begin{array}{*{20}c}
   Q & 0  \\
   0 & R  \\
\end{array}} \right]\\
& + \gamma \left[ {\begin{array}{*{20}c}
   A & B  \\
\end{array}} \right]^T (P_{11}^*  - P_{12}^* (P_{22}^*)^{ - 1} (P_{12}^*)^T )\left[ {\begin{array}{*{20}c}
   A & B  \\
\end{array}} \right]
\end{align*}
or more compactly,
\begin{align}
P^*  = \Lambda  + \gamma A(F^* )^T P^* A(F^* ).\label{eq:5}
\end{align}
We close this section by introducing some additional definitions and lemma. Throughout the paper, we will use the partition $P=
\begin{bmatrix}
   P_{11} & P_{12}\\
   P_{12}^T & P_{22}\\
\end{bmatrix}$ for any matrix  $P \in {\mathbb S}^{n+m}$, where $P_{11}\in {\mathbb S}^n,P_{12} \in {\mathbb R}^{n
\times m},\,P_{22} \in {\mathbb S}^m$.
For convenience, we introduce the set
\[
{\cal P}: = \left\{ {\left[ {\begin{array}{*{20}c}
   {P_{11} } & {P_{12} }  \\
   {P_{12}^T } & {P_{22} }  \\
\end{array}} \right] \in {\mathbb S}_+^{n + m} :P_{22}  \in {\mathbb S}_{ +  + }^m } \right\}
\]

\begin{lemma}[\cite{lee2018primal}]\label{lemma:key-result}
For $P \in {\cal P}$, it holds that
\[
A(F)^T PA(F) \succeq A( - P_{22}^{ - 1} P_{12}^T )^T PA( - P_{22}^{ - 1} P_{12}^T ),\quad \forall F \in {\mathbb R}^{m\times n}
\]

\end{lemma}

\section{Value Iteration}

In this section, we provide analysis of Q-value iteration (Q-VI). Define the Bellman operator
\[
{\cal T}(P) = \Lambda  + \gamma A( - P_{22}^{ - 1} P_{12}^T )^T PA( - P_{22}^{ - 1} P_{12}^T )
\]
Note that the operator is well defined only for $P \in {\cal P}$ because in this case, $P_{22}^{ - 1}$ exists in the definition of $\cal T$. In this paper, for $P=0$, the operator is defined as ${\cal T}(P) = \Lambda$. Then, ${\cal T}(0)\in {\cal P}$ holds because $\Lambda \in {\cal P}$. Moreover, $P^* \in {\cal P}$ as well.
\begin{lemma}[Positiveness of $\cal T$]
For any $P\in {\cal P} \cup \{ 0\}$, ${\cal T}(P)\in {\cal P}$ holds.
\end{lemma}

Using this matrix operator, Q-VI can be briefly summarized as in~\cref{algo:value-iteration}.
\begin{algorithm}[h]
\caption{Q-Value Iteration (Q-VI)}
\begin{algorithmic}[1]
\State Initialize $P_0 \in {\mathbb S}_+^{n+m}$.
\For{$k \in \{0,1,\ldots\}$}

\State Update $P_{k + 1}  ={\cal T}(P_k)$

\EndFor

\end{algorithmic}\label{algo:value-iteration}
\end{algorithm}

As a first step toward our goal, an important property of $\cal T$ is its monotonicity.
\begin{lemma}[Monotonicity of $\cal T$]\label{lemma:3}
$\cal T$ is ${\mathbb S}^{n+m}_+$ monotone on ${\cal P} \cup \{ 0 \}$, i.e., for $P', P \in {\cal P} \cup \{ 0 \}$, $P' \preceq P$ implies ${\cal T}(P') \preceq {\cal T}(P)$.
\end{lemma}
\begin{proof}
Suppose $\tilde P \preceq P$. Then,
\begin{align*}
{\cal T}(P) =& \Lambda  + \gamma A( - P_{22}^{ - 1} P_{12}^T )^T PA( - P_{22}^{ - 1} P_{12}^T )\\
 \succeq& \Lambda  + \gamma A( - P_{22}^{ - 1} P_{12}^T )^T \tilde PA( - P_{22}^{ - 1} P_{12}^T )\\
\succeq& \Lambda  + \gamma A( - \tilde P_{22}^{ - 1} \tilde P_{12}^T )^T \tilde PA( - \tilde P_{22}^{ - 1} \tilde P_{12}^T )\\
 = & {\cal T}(\tilde P)
\end{align*}
where the first inequality is due to the hypothesis, $\tilde P \preceq P$, and~\cref{lemma:key-result} is applied for the second inequality. This completes the proof.
\end{proof}

The analyze the convergence of~\cref{algo:value-iteration}, we now focus on the error matrix
\begin{align}
{\cal T}^k (P) - {\cal T}^k(P^*)= {\cal T}^k (P) - P^*.\label{eq:error}
\end{align}
In the following, an upper bound on the error is given in terms of the matrix bound on the semidefinite cone.
\begin{theorem}[Upper bound]\label{thm:1}
For any $P\in {\cal P} \cup \{ 0 \}$, we have
\[
{\cal T}^k (P) - P^* = {\cal T}^k (P) - {\cal T}^k(P^*) \preceq \gamma^k (A(F^*)^T )^k (P  - P^*)A(F^*)^k
\]
for all $k\geq 0$, and $\gamma^k (A(F^*)^T )^k (P  - P^*)A(F^*)^k \to 0$ as $k \to \infty$.
\end{theorem}
\begin{proof}
First of all, we can derive the following bounds
\begin{align*}
&{\cal T}(P) - {\cal T}(P^* )\\
 =& \gamma A( - P_{22}^{ - 1} P_{12}^T)^T PA( - P_{22}^{ - 1} P_{12}^T )  - \gamma A(F^*)^T P^* A(F^*)\\
\preceq & \gamma A(F^*)^T PA(F^*)  - \gamma A(F^*)^T P^* A(F^*)\\
=& \gamma A(F^*)^T (P - P^* )A(F^*)
\end{align*}
where the inequality is due to~\cref{lemma:key-result}.
For an induction argument, suppose that
\[
{\cal T}^{k - 1} (P) - {\cal T}^{k - 1} (P^* ) \preceq \gamma ^{k - 2} (A(F^* )^T )^{k - 2} (P - P^* )A(F^* )^{k - 2}
\]
holds. To proceed, let ${\cal T}^k (P) = P_k$ and $F_k  =  - P_{k,22}^{ - 1} P_{k,12}^T$.
Then, using $P^* = {\cal T}(P^*)$, we have
\begin{align*}
&{\cal T}^k (P) - {\cal T}^k (P^* )\\
=& \gamma A(F_{k - 1} )^T {\cal T}^{k - 1} (P)A(F_{k - 1} ) - \gamma A(F^* )^T {\cal T}^{k - 1} (P^* )A(F^* ))\\
\preceq & \gamma A(F^* )^T ({\cal T}^{k - 1} (P) - {\cal T}^{k - 1} (P^* ))A(F^* )\\
\preceq & \gamma ^{k - 1} (A(F^* )^T )^{k - 1} (P - P^* )A(F^* )^{k - 1}
\end{align*}
The desired conclusion is obtained by induction. Since $\gamma^{1/2} A(F^*)$ is Schur, $(\gamma^{1/2} A(F^*))^k \to 0$ as $k \to \infty$ (see~\cite[Theorem~5.6.12,~pp.~348]{horn2012matrix}), and the proof is completed.
\end{proof}

\cref{thm:1} only provides an upper bound on the error, ${\cal T}^k (P) - P^*$. On the other hand, its lower bound cannot be established in this way. However, under a special condition on the initial point $P$, a trivial lower bound can be found, and we can obtain an exponential convergence of the Q-value iteration to $P^*$. To proceed further, define the positive semidefinite cone
\[
{\cal C}(P): = \{ P' \in {\mathbb S}^{n + m} : P' \succeq P \}\subseteq {\cal P},
\]
which will play an important role in this paper.  We can prove that with $P\in {\cal C}(P^*)$, a complete exponential convergence can be obtained.
\begin{theorem}[Local convergence on the semidefinite cone]\label{thm:convergence1}
Suppose $P\in {\cal C}(P^*)$, then
\[
0 \preceq {\cal T}^k (P) - P^* \preceq  \gamma^k (A_{F^* }^T )^k (P  - P^*)A_{F^* }^k
\]
and $\gamma^k (A_{F^* }^T )^k (P  - P^*)A_{F^* }^k \to 0$ as $k \to \infty$.
\end{theorem}
\begin{proof}
The upper bound comes from~\cref{thm:1}, and the lower bound is due to the monotonicity in~\cref{lemma:3}.
This completes the proof.
\end{proof}

\cref{thm:convergence1} tells us that if the initial parameter satisfies $P\in {\cal C}(P^*)$, then the value iteration error is over bounded by a matrix which vanishes as $k \to \infty$, and under bounded by the zero matrix. A natural question arising here is if we can also derive the error bounds in terms of some matrix norm. To answer this question, some mathematical ingredients should be prepared. First of all, let us choose a proper matrix norm. For any $P \succ 0$, define the following norm:
\[
\left\| \cdot \right\|_P  = \sqrt {\lambda _{\max } ((\cdot)^T P (\cdot))}
\]
which is the induced matrix norm of the vector norm $\left\|  \cdot  \right\|_P  = \sqrt {( \cdot )^T P( \cdot )}$. An important property of the norm is the property called the monotonicity.
\begin{definition}[{\cite[pp.~57]{ciarlet1989introduction}}]
A matrix norm $\left\| \cdot \right\|$ is monotone if for any $A,B \in {\mathbb S}_+^{n+m}$ such that $A\succeq B$, $\left\| B\right\|\leq \left\| A \right\|$ holds.
\end{definition}

We can easily prove that $\left\| \cdot \right\|_P$ is monotone, presented in the following lemma. The proof is given in Appendix~I.
\begin{lemma}\label{lemma:monotone-norm}
For any $P \in {\mathbb S}_{++}^{n+m}$, the norm $\left\| \cdot \right\|_P$ is monotone.
\end{lemma}

For such a norm, $\left\| \cdot \right\|_P $, to meet our purpose, the matrix $P$ needs to be properly chosen. One can conclude that the matrix can be chosen as a Lyapunov matrix. In the sequel, we establish some results related to the Lyapunov inequality.
\begin{lemma}\label{lemma:1}
Suppose that $F \in {\cal F}_1$ so that $A(F)$ is Schur or equivalently, $\rho(A(F))<1$.
For any $\varepsilon>0$, there exists $P_\varepsilon \in {\mathbb S}_{++}^{n+m}$ such that the following Lyapunov inequality holds:
\[
A(F)^T P_\varepsilon A(F)  \preceq (\rho (A(F) ) + \varepsilon )^2 P_\varepsilon
\]
and
\[
\lambda _{\max } (P_\varepsilon) \le 1
\]
\end{lemma}

The proof of~\cref{lemma:1} is given in Appendix~II. We are now ready to derive a bound on the error~\eqref{eq:error}. It can be proven that the error bounds can be expressed as a matrix norm $\left\| \cdot \right\|_P $ with $P$ selected by~\cref{lemma:1}.
\begin{theorem}\label{thm:3}
For any $P \in {\cal C}(P^*)$, we have
\begin{align}
\| {{\cal T}(P) - P^*} \|_{P_\varepsilon ^* }  \le  (\rho (\gamma^{1/2} A(F^*) ) + \varepsilon )^2 \| {P - P^* } \|_{P_\varepsilon ^* } ,\label{eq:2}
\end{align}
i.e., it is pseudo contraction over ${\cal C}(P^*)$, where $P_\varepsilon ^* \in {\mathbb S}^{n+m}_{++}$ is a matrix satisfying the conditions in~\cref{lemma:1} with $\gamma^{1/2} A(F^*)$. Moreover, for all $P \in {\cal C}(P^*)$ and $k\geq 0$, we have
\begin{align}
\| {{\cal T}^k (P) - P^* } \|_{P_\varepsilon ^* }\le (\rho (\gamma^{1/2} A(F^*)) + \varepsilon )^{2k} \| {P - P^* } \|_{P_\varepsilon ^* },\label{eq:3}
\end{align}
for any $\varepsilon >0$ such that $\rho ( \gamma^{1/2} A(F^*)) + \varepsilon  < 1$.
\end{theorem}
\begin{proof}
We first conclude that
\begin{align}
\| \gamma^{1/2} A(F^*)\|_{P_\varepsilon ^* }  =& \sqrt {\lambda _{\max } (\gamma^{1/2} A(F^*)^T P_\varepsilon ^* \gamma^{1/2} A(F^*) )}\nonumber\\
\le& (\rho (\gamma^{1/2} A(F^*)) + \varepsilon )\sqrt {\lambda _{\max } (P_\varepsilon^*)}\nonumber\\
\le& \rho (\gamma^{1/2} A(F^*)) + \varepsilon\label{eq:1}
\end{align}
where we used $\gamma^{1/2} A(F^*)^T P_\varepsilon ^*\gamma^{1/2}  A(F^*)\preceq (\rho (\gamma^{1/2} A(F^*)) + \varepsilon )^2 P_\varepsilon ^*$ in the first inequality, and $\lambda _{\max } (P_\varepsilon ^* ) \le 1$ in the second inequality. On the other hand, taking the norm, $\left\| \cdot \right\|_{P_\varepsilon ^* }$, to the inequality in~\cref{thm:convergence1}, we have
\begin{align}
\left\| {{\cal T}(P) - {\cal T}(P^* )} \right\|_{P_\varepsilon ^* }  \le& \| {\gamma^{1/2} A(F^*)^T (P - P^* )\gamma^{1/2} A(F^*) }\|_{P_\varepsilon ^*}\nonumber \\
 \le& \left\| {P - P^* } \right\|_{P_\varepsilon ^* } \|\gamma^{1/2} A(F^*)\|_{P_\varepsilon ^* }^2 \nonumber\\
 \le& \left\| {P - P^* } \right\|_{P_\varepsilon ^* } (\rho (\gamma^{1/2} A(F^*)) + \varepsilon )^2\label{eq:8}
\end{align}
where the last inequality comes from~\eqref{eq:1}, which is~\eqref{eq:2}. Recursively combining~\eqref{eq:8} yields~\eqref{eq:3}. This completes the proof.
\end{proof}

The bound in~\cref{thm:3} can be readily expressed in terms of the spectral norm, which is summarized below.
\begin{corollary}
For any $P \in {\cal C}(P^*)$, we have
\begin{align*}
&\| {T^k (P) - P^* } \|_2 \\
\le & \frac{{\lambda _{\max } (P_\varepsilon ^* )}}{{\lambda _{\min } (P_\varepsilon ^* )}}(\rho (\gamma ^{1/2} A(F^* )) + \varepsilon )^{2k} \left\| {P - P^* } \right\|_2,
\end{align*}
for any $\varepsilon >0$ such that $\rho ( \gamma^{1/2} A(F^*)) + \varepsilon  < 1$, where $P_\varepsilon ^* \in {\mathbb S}^{n+m}_{++}$ is a matrix satisfying the conditions in~\cref{lemma:1}.
\end{corollary}

The convergence of~\cref{thm:3} requires that the initial parameter $P$ is within ${\cal C}(P^*)$ with an unknown $P^*$. In the general case where $P$ may not be in $ {\cal C}(P^*)$, such a lower bound is hard to be established. Instead, we can obtain a bound with asymptotic convergence in the following result.
\begin{theorem}[Global convergence]\label{thm:4}
For any $P \in {\cal P} \cup \{ 0 \}$, we have
\begin{align*}
 {\cal T}^k (0) - {\cal T}^k (P^* )\preceq & {\cal T}^k (P) - {\cal T}^k (P^* )\\
 \preceq & \gamma ^k (A(F^* )^T )^k (P - P^* )A(F^* )^k,\quad \forall k \geq 0,
\end{align*}
where $\gamma ^k (A(F^* )^T )^k (P - P^* )A(F^* )^k \to 0$ and ${\cal T}^k (0) -{\cal T}^k (P^* ) \to 0$ as $k\to \infty$.
\end{theorem}
\begin{proof}
The upper bound is due to~\cref{thm:1}. For the lower bound, note that from the monotonicity of $\cal T$ in~\cref{lemma:3}, ${\cal T}(0) \preceq {\cal T}(P)$, and hence, ${\cal T}^k(0) - {\cal T}^k(P^*) \preceq {\cal T}^k(P) - {\cal T}^k(P^*)$. It remains to prove that $\gamma ^k (A(F^* )^T )^k (P - P^* )A(F^* )^k \to 0$ and ${\cal T}^k (0) - {\cal T}^k (P^* ) \to 0$ as $k\to \infty$. The former is true because $\gamma^{1/2} A(F^* )$ is Schur (see~\cite[Theorem~5.6.12,~pp.~348]{horn2012matrix}). To prove ${\cal T}^k (0)\to {\cal T}^k (P^* ) = P^*$, note that $0 \preceq {\cal T}(0)$. From the monotonicity of ${\cal T}$, one concludes ${\cal T}^k(0) \preceq {\cal T}^{k+1}(0)$, and hence, ${\cal T}^k(0)$ is monotonically non-decreasing. Moreover, from the upper bound, ${\cal T}^k (0) \preceq {\cal T}^k (P) \preceq \gamma ^k (A(F^* )^T )^k (P - P^* )A(F^* )^k + P^*$, we conclude that ${\cal T}^k(0)\to S$ as $k \to \infty$ for some matrix $S  \in {\mathbb S}_+^{n+m}$ such that $T(S)=S$, implying that $S = P^*$.  This completes the proof.
\end{proof}

\cref{thm:4} offers a global convergence result of the error in terms of upper and lower bounds on the semidefinite cone. The upper bound is applied to the general case, and it can provide a finite-time analysis. On the other hand, the convergence of the lower bound is asymptotic. Overall convergence in this case is dominated by the asymptotic behavior of the lower bound. Besides, the upper bound can be further analyzed, and can be proven to converge exponentially after a certain number $N\geq 0$ of iterations.
\begin{corollary}
Consider any $P\in {\cal P} \cup \{ 0 \}$. Then, for any $\varepsilon >0$, there exists an integer $N\geq 0$ such that
\[
{\cal T}^k (P) - {\cal T}^k (P^* ) \preceq |\lambda _{\max } (P - P^* )|(\rho (\gamma ^{1/2} A(F^* )) + \varepsilon )^{2k} I
\]
for all $k \geq N$, where $N$ is such that
\[
\| {\gamma ^{k/2} A(F^* )^k }\|_2^{1/k}  \le \rho (\gamma ^{1/2} A(F^* )) + \varepsilon ,\quad \forall k \ge N
\]
\end{corollary}
\begin{proof}
From~\cref{thm:4}, one gets
\begin{align*}
{\cal T}^k (P) - {\cal T}^k (P^* ) \preceq&  \gamma ^k (A(F^* )^T )^k (P - P^* )A(F^* )^k\\
\preceq &  |\lambda _{\max } (P - P^* )| \| {\gamma ^{k/2} A(F^* )^k } \|_2^2 I\\
=& |\lambda _{\max } (P - P^* )|(\| {\gamma ^{k/2} A(F^* )^k } \|_2^{1/k} )^{2k} I.
\end{align*}
From the Gelfand's formula~\cite[Corollary~5.6.14,~pp.~349]{horn2012matrix}, there exists a finite $N\geq 0$ such that
\[
\| {\gamma^{k/2} A(F^*)^k } \|_2^{1/k}  \le \rho (\gamma^{1/2} A(F^*)) + \varepsilon ,\quad \forall k \ge N.
\]
Plugging the bounds into the previous inequalities yields the desired conclusion.
\end{proof}

Although~\cref{thm:4} does not provide a finite-time lower bound, we can prove that after a sufficient number, $N$, of iterations, the lower bound also converges exponentially fast. The result is presented in the sequel, and the proof is given in Appendix~III.
\begin{proposition}\label{thm:7}
Consider any $P\in {\cal P} \cup \{ 0 \}$. Then,  for any $\varepsilon_1,\varepsilon_2 >0$, there exist integers $N_1,N_2\geq 0$ and a constant, $\eta>0$, such that
\[
{\cal T}^{N_1  + k} (P) - {\cal T}^{N_1  + k} (P^* )\succeq  - \eta \{ (\rho (\gamma ^{1/2} A(F^* )) + \varepsilon _2 )^k  + \varepsilon _1 \} ^2 I
\]
for all $k \geq N_2$.
\end{proposition}

\cref{thm:7} suggests that although the lower bound on the error can progress with a sublinear speed, it could eventually converge with linear rates. This behavior will be demonstrated in the example section. Finally, in this section, an analysis of Q-VI has been established in terms of the matrix bounds and matrix norm bounds. Especially, under the condition, $P \in {\cal C}(P^*)$, on the initial point, an exponential convergence has been derived. Before closing this section, we briefly discuss the value function counterpart of Q-VI~(\cref{algo:value-iteration}), which is the well-known Riccati recursion: for all $k \in \{0,1,\ldots\}$
\[
X_{k + 1}  = \gamma A^T X_k A - \gamma A^T X_k B(R + \gamma B^T X_k B)^{ - 1} \gamma B^T X_k A + Q
\]
with any initial $X_0 \in {\mathbb S}_+^n$. For this case, noting that
\begin{align*}
X_k  - X^*  = \left[ {\begin{array}{*{20}c}
   I  \\
   0  \\
\end{array}} \right]^T \{ {\cal T}^k (P) - {\cal T}^k (P^* )\} \left[ {\begin{array}{*{20}c}
   I  \\
   0  \\
\end{array}} \right]
\end{align*}
we can conclude that all the results corresponding to Q-VI can be directly obtained for the convergence of the VI as well.

\section{Policy Iteration}

The policy iteration, summarized in~\cref{algo:policy-iteration}, is another class of the dynamic programming algorithm which iterates policies together with values. Especially, we will consider a policy iteration based on the Q-function, which will be called Q-PI throughout this paper.
\begin{algorithm}[h]
\caption{Q-function based policy iteration (Q-PI)}
\begin{algorithmic}[1]
\State Initialize $F_0 \in {\cal F}_\gamma$.
\For{$k \in \{0,1,\ldots\}$}
\State Solve for $P_k$ the linear equation
\begin{align}
P_k = \Lambda  + \gamma A(F_k)^T P_k A(F_k)\label{eq:9}
\end{align}

\State Update $F_{k + 1}  =  - (P^k_{22})^{ - 1} (P^k_{12})^T$

\EndFor

\end{algorithmic}\label{algo:policy-iteration}
\end{algorithm}
Note that in Q-PI, an additional information is required initially, namely, the initially stabilizing gain $F_0 \in {\cal F}_\gamma$ for $(\gamma^{1/2} A,\gamma^{1/2} B)$. At each iteration, one needs to solve the Bellman equation in~\eqref{eq:9}, which is linear in $P_k$, to evaluate the given gain $F_k$, which can be achieved by using the following recursion
\begin{align*}
P_{k,i+1} = \Lambda  + \gamma A(F_k)^T P_{k,i} A(F_k)
\end{align*}
for $i\in \{0,1,\ldots \}$ with $P_0 \in {\mathbb S}_+^{n+m}$. The global exponential convergence of the iteration can be easily proved using the same lines as in the proof of~\cref{thm:3}, so omitted here. Therefore, we will focus on the outer iteration given in~\cref{algo:policy-iteration}. To analyze Q-PI, define the mapping ${\cal H}(P)$ for any $P \in {\cal P}$ such that
\[
{\cal H}(P) = S,
\]
where
\[
S = \Lambda  + \gamma A( - P_{22}^{ - 1} P_{12}^T )^T SA( - P_{22}^{ - 1} P_{12}^T )
\]
Then, Q-PI in~\cref{algo:policy-iteration} can be viewed as the recursion ${\cal H}^k (P) = P_k$. To proceed, for any given $F \in {\cal F}_\gamma$, define the mapping
\[
{\cal L}_F( \cdot) = \Lambda  + \gamma A(F)^T ( \cdot )A(F),
\]
which will play an important role in the analysis of Q-PI. Based on these definitions, some useful properties of $\cal H$ are summarized below.
\begin{lemma}\label{lemma:2}
Suppose that for any given $F \in {\cal F}_\gamma$, $P$ satisfies $P = \Lambda  + \gamma A(F)^T PA(F)$. Then, $P \succeq {\cal H}(P)$ holds.
\end{lemma}
\begin{proof}
Based on the hypothesis, $P = \Lambda  + \gamma A(F)^T PA(F)$, observe that
\begin{align}
P =& \Lambda  + \gamma A(F)^T PA(F)\nonumber\\
\succeq& \Lambda  + \gamma A( - P_{22}^{ - 1} P_{12}^T )^T PA( - P_{22}^{ - 1} P_{12}^T )\nonumber\\
=& {\cal L}_{- P_{22}^{ - 1} P_{12}}(P),\label{eq:4}
\end{align}
where the inequality is due to~\cref{lemma:key-result}. Since ${\cal L}_{- P_{22}^{ - 1} P_{12}^T}$ is trivially monotone, we can prove from~\eqref{eq:4} that
\[
P \succeq {\cal L}_{- P_{22}^{ - 1} P_{12}^T}(P) \succeq {\cal L}_{- P_{22}^{ - 1} P_{12}^T}^2 (P) \succeq  \cdots  \succeq {\cal L}_{- P_{22}^{ - 1} P_{12}^T}^k (P)
\]
Therefore ${\cal L}_{- P_{22}^{ - 1} P_{12}^T}^k(P)$ is monotonically non-increasing. Moreover, it is bounded below by $0$, ${\cal L}_{- P_{22}^{ - 1} P_{12}^T}^k (P)\succeq 0$, and hence, converges to some $S\in {\mathbb S}_+^{n+m}$, $\mathop {\lim }\limits_{k \to \infty } {\cal L}_{- P_{22}^{ - 1} P_{12}^T}^k (P) = S \in {\mathbb S}_+^{n+m}$, such that ${\cal L}_{- P_{22}^{ - 1} P_{12}^T}(S) = S$. Therefore,
\[
P \succeq \mathop {\lim }\limits_{k \to \infty } {\cal L}_{- P_{22}^{ - 1} P_{12}^T}^k (P) = {\cal H}(P)
\]
which completes the proof.
\end{proof}

Note that $\cal H$ is a basic operator used for Q-PI, which corresponds to $\cal T$ in Q-VI. Similarly to $\cal T$, we can prove that $\cal H$ is monotone as well.
\begin{lemma}[Monotonicity of $\cal H$]
$\cal H$ is ${\mathbb S}^{n+m}_+$ monotone on ${\cal P}$, i.e., for $P', P \in {\cal P}$, $P' \preceq P$ implies ${\cal H}(P') \preceq {\cal H}(P)$.
\end{lemma}
\begin{proof}
For any $P,P' \in {\mathbb S}_+^{n+m}$ such that $P\preceq P'$, the monotonicity of ${\cal L}_{- P_{22}^{ - 1} P_{12}^T}$ yields
\[
{\cal L}_{- P_{22}^{ - 1} P_{12}^T}^k (P) \preceq {\cal L}_{- P_{22}^{ - 1} P_{12}^T}^k (P').
\]
Then, by taking the limit on both sides, we have ${\cal H}(P) \preceq {\cal H}(P')$.
\end{proof}

Based on the aforementioned results, one can obtain an upper bound on the error, ${\cal H}^k (P ) - P^*$, in terms of the matrix inequality similar to Q-VI.
\begin{theorem}[Upper bound]\label{thm:5}
Suppose that for any given $F \in {\cal F}_\gamma$, $P$ satisfies $P = \Lambda  + \gamma A(F)^T PA(F)$. Then
\[
{\cal H}^k (P ) - P^* \preceq \gamma ^k (A(F^* )^T )^k (P  - P^* )A(F^* )^k
\]
\end{theorem}
\begin{proof}
Suppose that ${\cal H}(P) = S  =\Lambda +  \gamma A( - P_{22}^{ - 1} P_{12}^T )^T SA( - P_{22}^{ - 1} P_{12}^T )$. Using this relation, we have
\begin{align*}
&{\cal H}(P) - {\cal H}(P^* )\\
 =& S - \gamma A(F^* )^T P^* A(F^* )\\
 =& \gamma A( - P_{22}^{ - 1} P_{12}^T )^T SA( - P_{22}^{ - 1} P_{12}^T ) - \gamma A(F^* )^T P^* A(F^* )\\
 \preceq& \gamma A( - P_{22}^{ - 1} P_{12}^T )^T PA( - P_{22}^{ - 1} P_{12}^T ) - \gamma A(F^* )^T P^* A(F^* )\\
 \preceq& \gamma A(F^* )^T PA(F^* ) - \gamma A(F^* )^T P^* A(F^* )\\
 =& \gamma A(F^* )^T (P - P^* )A(F^* ),
\end{align*}
where the first inequality uses ${\cal H}(P) = S \preceq P$ from~\cref{lemma:2}, and the second inequality is due to~\cref{lemma:key-result}. The desired result is obtained using an induction argument similar to the proof of~\cref{thm:convergence1}.
\end{proof}

As in the previous section, an identical upper bound can be obtained for Q-VI. One may imagine that the lower bound can be obtained under the special initial point, $P \in {\cal C}(P^*)$ as in Q-VI case. However, in the policy iteration case, we can prove that the initial $P$ always satisfies $P \in {\cal C}(P^*)$. The benefit comes from the additional knowledge on the initial gain $F_0$ which is stabilizing for $(\gamma^{1/2} A,\gamma^{1/2} B)$. Therefore, a global exponential convergence can be derived for Q-VI as follows.
\begin{theorem}[Global convergence]\label{thm:6}
Suppose that for any given $F \in {\cal F}_\gamma$, $P$ satisfies $P = \Lambda  + \gamma A(F)^T PA(F)$. Then, $P \in {\cal C}(P^*)$, and we have
\[
0 \preceq {\cal H}^k (P ) - P^* \preceq \gamma ^k (A(F^* )^T )^k (P  - P^* )A(F^* )^k,
\]
for all $k \geq 0$. Moreover, we have
\begin{align*}
\| {{\cal H}^k (P) - P^* } \|_{P_\varepsilon ^* }\preceq (\rho (\gamma^{1/2} A(F^*)) + \varepsilon )^{2k} \| {P - P^* } \|_{P_\varepsilon ^* },
\end{align*}
for all $k\geq 0$ and for any $\varepsilon >0$ such that $\rho ( \gamma^{1/2} A(F^*)) + \varepsilon  < 1$ and $P_\varepsilon ^* \in {\mathbb S}^{n+m}_{++}$ is a matrix satisfying the conditions in~\cref{lemma:1} with $\gamma^{1/2} A(F^*)$.
\end{theorem}
\begin{proof}
The proof of the error bounds follows the same lines as in the proof of~\cref{thm:convergence1} and~\cref{thm:3}, so it is omitted in this paper. We only prove the fact that under the initialization scheme in~\cref{algo:policy-iteration}, $P \in {\cal C}(P^*)$ holds. The initial point, $P$, of the policy iteration satisfies $P = \Lambda  + \gamma A(F)^T PA(F)$ for any given $F \in {\cal F}_\gamma$. Define $z(k;F,x,u)\in {\mathbb R}^{n+m}$ as the augmented state trajectory, $\left[ {\begin{array}{*{20}c}
   {x(k)}  \\
   {u(k)}  \\
\end{array}} \right], k\geq 0$, with initial value $\left[ {\begin{array}{*{20}c}
   {x(0)}  \\
   {u(0)}  \\
\end{array}} \right] = \left[ {\begin{array}{*{20}c}
   x  \\
   u  \\
\end{array}} \right]$ and under the input $u(k) = F x(k)$. Then, from the definition of $F^*$ and using the relation
\[
P = \sum\limits_{k = 0}^\infty  {\gamma ^k (A(F)^T )^k \Lambda A(F)},
\]
we have
\begin{align*}
J(x;F^* ) =& \min _{u \in {\mathbb R}^{m \times n} } \sum\limits_{k = 0}^\infty  {\gamma ^k z(k;F^* ,x,u)^T \Lambda z(k;F^* ,x,u)}\\
 =& \min _{u \in {\mathbb R}^{m \times n} } Q^* (x,u)\\
 =& \min _{u \in {\mathbb R}^{m \times n} } \left[ {\begin{array}{*{20}c}
   x  \\
   u  \\
\end{array}} \right]^T P^* \left[ {\begin{array}{*{20}c}
   x  \\
   u  \\
\end{array}} \right]\\
\le& \sum\limits_{k = 0}^\infty  {\gamma ^k z(k;F,x,\tilde u)^T \Lambda z(k;F,x,\tilde u)}\\
 =& \left[ {\begin{array}{*{20}c}
   x  \\
   {\tilde u}  \\
\end{array}} \right]^T P\left[ {\begin{array}{*{20}c}
   x  \\
   {\tilde u}  \\
\end{array}} \right]
\end{align*}
for any $\tilde u\in {\mathbb R}^{m}$. Therefore, it follows that
\[
\left[ {\begin{array}{*{20}c}
   x  \\
   {\hat u}  \\
\end{array}} \right]^T P^* \left[ {\begin{array}{*{20}c}
   x  \\
   {\hat u}  \\
\end{array}} \right] \le \left[ {\begin{array}{*{20}c}
   x  \\
   {\tilde u}  \\
\end{array}} \right]^T P\left[ {\begin{array}{*{20}c}
   x  \\
   {\tilde u}  \\
\end{array}} \right]
\]
for any $\hat u,\tilde u\in {\mathbb R}^{m}$, or equivalently, $P^* \preceq P$. This completes the proof.
\end{proof}

Contrary to Q-VI, Q-PI always guarantees $P \succeq P^*$, and hence, the exponential convergence in~\cref{thm:6} always holds. Compared to Q-VI, this improvement comes from the additional information on the initially stabilizing gain $F_0 \in {\cal F}_\gamma$. If $F_0 \in {\cal F}_\gamma$ is initially given, then Q-VI can be improved as well. In particular, one can develop the following two phases algorithm: 1) For initial $F \in {\cal F}_\gamma$, find $\tilde P \in {\mathbb S}_{++}^{n+m}$ such that $\tilde P = \Lambda  + I + \gamma A(F)^T \tilde PA(F) =: {\cal D}(\tilde P)$. One can prove that $P \succeq P^*+I$, and ${\cal D}^k (P)$ converges to $\tilde P$ as $k \to \infty$ exponentially fast. Therefore, there exists a finite $k>0$ such that ${\cal D}^k (P) \in {\cal C}(P^*)$. Since there exists an explicit gap between $\tilde P$ and $P^*$, we can find a lower bound on the number of iteration such that ${\cal D}^k (P) \in {\cal C}(P^*)$. 2) Once $P_k = {\cal D}^k(P)$ enters ${\cal C}(P^*)$, then run Q-VI with $P_k$ as an initial parameter. The two phases process guarantees the exponential convergence to $P^*$.

Although the convergence is given in terms of the Q-function parameter $P_k$, if it converges to $P^*$, this implies that the corresponding gain $F_{k + 1}  =  - (P^k_{22})^{ - 1} (P^k_{12})^T$ also converges to $F^*$. Therefore, results in this section also establish the convergence of the gains.
\begin{algorithm}[h]
\caption{Value function based policy iteration (PI)}
\begin{algorithmic}[1]
\State Initialize $F_0 \in {\cal F}_\gamma$.
\For{$k \in \{0,1,\ldots\}$}
\State Solve for $X_k$ the linear equation
\[
X_{k}  = \gamma A^T X_k A - F_k^T (R + \gamma B^T X_k B) F_k + Q
\]

\State Update $F_{k + 1}  =  -(R + \gamma B^T X_k B)^{ - 1} \gamma B^T X_k A$

\EndFor

\end{algorithmic}\label{algo:policy-iteration2}
\end{algorithm}
Finally, the standard value function-based policy iteration (PI) is given in~\cref{algo:policy-iteration2}. As in the previous section, the analysis of Q-PI can be directly extended to PI in~\cref{algo:policy-iteration2} by noting the identity
\begin{align*}
X_k  - X^*  = \left[ {\begin{array}{*{20}c}
   I  \\
   0  \\
\end{array}} \right]^T \{ {\cal H}^k (P) - {\cal H}^k (P^* )\} \left[ {\begin{array}{*{20}c}
   I  \\
   0  \\
\end{array}} \right]
\end{align*}
In the sequel, an example is studied to demonstrate the validity of the analysis given throughout the paper.

\section{Example}
Consider the randomly generated system $(A,B)$
\[
A = \left[ {\begin{array}{*{20}c}
   {0.4527} & {0.9648}  \\
   {0.9521} & {0.6309}  \\
\end{array}} \right],\quad B = \left[ {\begin{array}{*{20}c}
   {0.2871}  \\
   {0.5994}  \\
\end{array}} \right]
\]
and
\[
Q = 0.1\left[ {\begin{array}{*{20}c}
   1 & 1  \\
\end{array}} \right]\left[ {\begin{array}{*{20}c}
   1  \\
   1  \\
\end{array}} \right],\quad R = 100,\quad \gamma =0.9.
\]

The optimal $P^*$ is
\[
P^*  = \left[ {\begin{array}{*{20}c}
   {2604.8} & {2877.2} & {1643.4}  \\
   {2877.2} & {3178.1} & {1815.3}  \\
   {1643.4} & {1815.3} & {2036.9}  \\
\end{array}} \right]
\]
and the corresponding optimal gain is
\[
F^*  = \left[ {\begin{array}{*{20}c}
   {-0.8068} & {-0.8912}  \\
\end{array}} \right]
\]

The spectral radius of $\gamma ^{1/2} A(F^* )$ is $\rho (\gamma ^{1/2} A(F^* )) = 0.7006<1$.
We run Q-VI in~\cref{algo:value-iteration} with the two initial parameters $P_0  = \lambda _{\min } (P^* )I$ so that $P_0\preceq P^*$, and $P_0  = \lambda _{\max } (P^* )I$ so that $P_0\succeq P^*$, where $\lambda _{\min } (P^* ) = 0.0005$ and $\lambda _{\max } (P^* ) = 6992.8$. The evolution of the error, $\left\| {P_k  - P^* } \right\|_2$, of Q-VI for different initial parameter $P_0  = \lambda _{\min } (P^* )I$ (blue line) and $P_0  = \lambda _{\max} (P^* )I$ (red line) are depicted in~\cref{fig:1}(a), and its log-scale plot is given in~\cref{fig:1}(b).
\begin{figure}[t]
\centering\subfigure[]{\includegraphics[width=7cm,height=5cm]{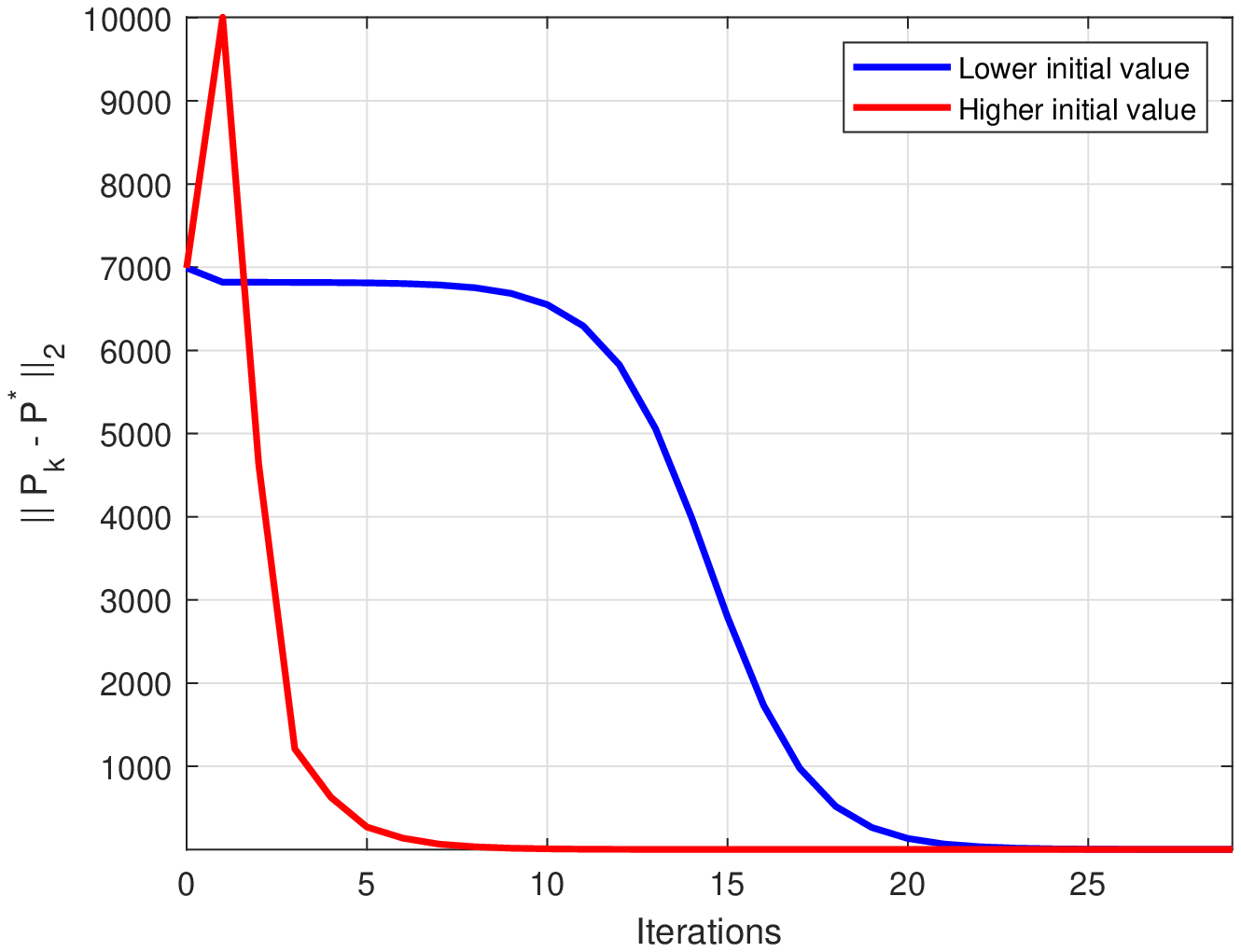}}
\centering\subfigure[]{\includegraphics[width=7cm,height=5cm]{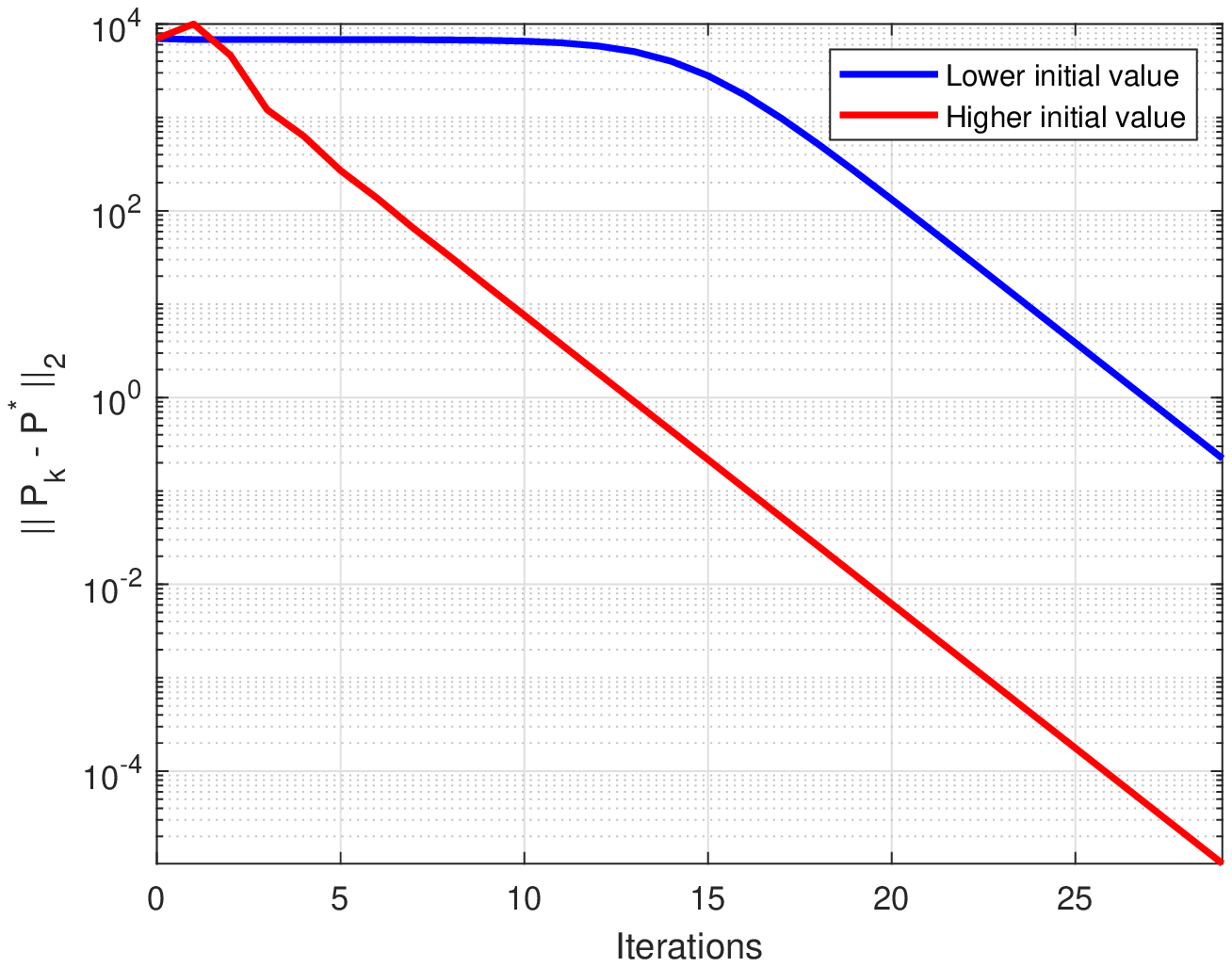}}
\caption{(a) Evolution of error, $\left\| {P_k  - P^* } \right\|_2$, of the Q-value iteration for different initial parameter $P_0  = \lambda _{\min } (P^* )I$ (blue line) and $P_0  = \lambda _{\max} (P^* )I$ (red line). (b) Evolution of error, $\left\| {P_k  - P^* } \right\|_2$, of the Q-value iteration in log-scale for different initial parameter $P_0  = \lambda _{\min } (P^* )I$ (blue line) and $P_0  = \lambda _{\max} (P^* )I$ (red line).}\label{fig:1}
\end{figure}
The figures suggest that the evolution of $\left\| {P_k  - P^* } \right\|_2$ with $P_0 \preceq P^*$ has sublinear convergence, while it has linear (or exponential) convergence with $P_0 \succeq P^*$, which match with the proposed analysis.

We also investigate the evolution of $[P_k  - P^*]_+$, which denotes the projection of the error onto the positive semidefinite cone, and $[P_k  - P^*]_-$, the projection of the error onto the negative semidefinite cone. If $P_0 \preceq P^*$ initially, then $P_k  = {\cal T}^k (P_0 ) \preceq P^*  = {\cal T}^k (P^* )$ for any $k\succeq 0$, and hence, $\left\| {[P_k  - P^*]_+ } \right\|_2 = 0$ for all $k\geq  0$. Therefore, we consider an indefinite initial point by setting
\[
P_0  = \frac{1}{2}(\lambda _{\min } (P^* ) + \lambda _{\max } (P^* ))\left[ {\begin{array}{*{20}c}
   1 & 1 & 1  \\
   1 & 1 & 1  \\
   1 & 1 & 1  \\
\end{array}} \right]
\]
so that neither $P_0  \preceq P^*$ nor $P_0 \succeq P^*$.
\begin{figure}[t]
\centering\includegraphics[width=7cm,height=5cm]{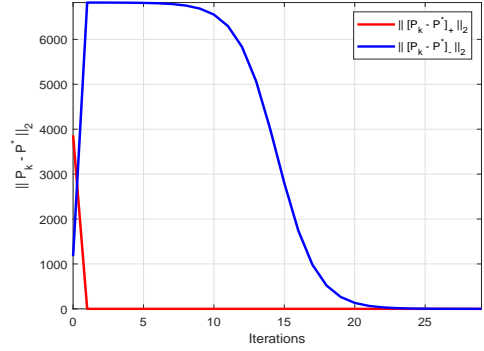}
\caption{Evolutions of $\left\| [{P_k  - P^* }]_+ \right\|_2$ (red line) and $\left\| [{P_k  - P^* }]_- \right\|_2$ (blue line) of the Q-value iteration.}\label{fig:2}
\end{figure}
The evolutions of $\left\| [{P_k  - P^* }]_+ \right\|_2$ (red line) and $\left\| [{P_k  - P^* }]_- \right\|_2$ (blue line) of Q-VI are shown in~\cref{fig:2}, which suggests that the positive semidefinite part, $\left\| [{P_k  - P^* }]_+ \right\|_2$, converges faster with an exponential rate, while the negative semidefinite part, $\left\| [{P_k  - P^* }]_- \right\|_2$, converges with a sublinear rate.
These results empirically demonstrate the theoretical analysis in this paper.

\section*{Conclusion}
In this paper, we have studied the convergence of Q-VI and Q-PI for discrete-time LTI systems. Bounds on errors have been given in terms of both matrix inequalities and matrix norm. In particular, we have proved that Q-VI exponentially converges to the optimal solution if the initial parameter lies in a certain semidefinite cone. A simple analysis of convergence in general cases has also been presented. These results have been then extended to Q-PI. Finally, an example has been given to illustrate the validity of the proposed analysis. Potential future works include analysis for generalized dynamic programming with errors incurred in each update step, extensions to switching linear systems, and analysis for approximate dynamic programming and reinforcement learning algorithms.
\bibliographystyle{IEEEtran}
\bibliography{reference}

\begin{thebibliography}{10}
\providecommand{\url}[1]{#1}
\csname url@samestyle\endcsname
\providecommand{\newblock}{\relax}
\providecommand{\bibinfo}[2]{#2}
\providecommand{\BIBentrySTDinterwordspacing}{\spaceskip=0pt\relax}
\providecommand{\BIBentryALTinterwordstretchfactor}{4}
\providecommand{\BIBentryALTinterwordspacing}{\spaceskip=\fontdimen2\font plus
\BIBentryALTinterwordstretchfactor\fontdimen3\font minus
  \fontdimen4\font\relax}
\providecommand{\BIBforeignlanguage}[2]{{%
\expandafter\ifx\csname l@#1\endcsname\relax
\typeout{** WARNING: IEEEtran.bst: No hyphenation pattern has been}%
\typeout{** loaded for the language `#1'. Using the pattern for}%
\typeout{** the default language instead.}%
\else
\language=\csname l@#1\endcsname
\fi
#2}}
\providecommand{\BIBdecl}{\relax}
\BIBdecl

\bibitem{bellman1965dynamic}
R.~Bellman and R.~E. Kalaba, \emph{Dynamic programming and modern control
  theory}, 1965, vol.~81.

\bibitem{bertsekas1996neuro}
D.~P. Bertsekas and J.~N. Tsitsiklis, \emph{Neuro-dynamic programming}.\hskip
  1em plus 0.5em minus 0.4em\relax Athena Scientific Belmont, MA, 1996.

\bibitem{bertsekas2005dynamic}
D.~P. Bertsekas, \emph{Dynamic Programming and Optimal Control}, 4th~ed.\hskip
  1em plus 0.5em minus 0.4em\relax Nashua, MA: Athena Scientific, 2005, vol.~1.

\bibitem{lewis2012optimal}
F.~L. Lewis, D.~Vrabie, and V.~L. Syrmos, \emph{Optimal control}.\hskip 1em
  plus 0.5em minus 0.4em\relax John Wiley \& Sons, 2012.

\bibitem{caines1970discrete}
P.~E. Caines and D.~Q. Mayne, ``On the discrete time matrix riccati equation of
  optimal control,'' \emph{International Journal of Control}, vol.~12, no.~5,
  pp. 785--794, 1970.

\bibitem{heydari2014revisiting}
A.~Heydari, ``Revisiting approximate dynamic programming and its convergence,''
  \emph{IEEE transactions on cybernetics}, vol.~44, no.~12, pp. 2733--2743,
  2014.

\bibitem{al2008discrete}
A.~Al-Tamimi, F.~L. Lewis, and M.~Abu-Khalaf, ``Discrete-time nonlinear hjb
  solution using approximate dynamic programming: Convergence proof,''
  \emph{IEEE Transactions on Systems, Man, and Cybernetics, Part B
  (Cybernetics)}, vol.~38, no.~4, pp. 943--949, 2008.

\bibitem{zhang2009value}
W.~Zhang, J.~Hu, and A.~Abate, ``On the value functions of the discrete-time
  switched lqr problem,'' \emph{IEEE Transactions on Automatic Control},
  vol.~54, no.~11, pp. 2669--2674, 2009.

\bibitem{zhang2011infinite}
------, ``Infinite-horizon switched lqr problems in discrete time: A suboptimal
  algorithm with performance analysis,'' \emph{IEEE Transactions on Automatic
  Control}, vol.~57, no.~7, pp. 1815--1821, 2011.

\bibitem{lincoln2006relaxing}
B.~Lincoln and A.~Rantzer, ``Relaxing dynamic programming,'' \emph{IEEE
  Transactions on Automatic Control}, vol.~51, no.~8, pp. 1249--1260, 2006.

\bibitem{sutton1998reinforcement}
R.~S. Sutton and A.~G. Barto, \emph{Reinforcement learning: {A}n
  introduction}.\hskip 1em plus 0.5em minus 0.4em\relax MIT Press, 1998.

\bibitem{lewis2009reinforcement}
F.~L. Lewis and D.~Vrabie, ``Reinforcement learning and adaptive dynamic
  programming for feedback control,'' \emph{Circuits and Systems Magazine,
  IEEE}, vol.~9, no.~3, pp. 32--50, 2009.

\bibitem{bradtke1994adaptive}
S.~J. Bradtke, B.~E. Ydstie, and A.~G. Barto, ``Adaptive linear quadratic
  control using policy iteration,'' in \emph{American Control Conference,
  1994}, vol.~3, 1994, pp. 3475--3479.

\bibitem{watkins1992q}
C.~J. Watkins and P.~Dayan, ``Q-learning,'' \emph{Machine learning}, vol.~8,
  no. 3-4, pp. 279--292, 1992.

\bibitem{lee2018primal}
D.~Lee and J.~Hu, ``Primal-dual {Q}-learning framework for {LQR} design,''
  \emph{IEEE Transactions on Automatic Control}, vol.~64, no.~9, pp.
  3756--3763, 2018.

\bibitem{geromel1998static}
J.~C. Geromel, C.~De~Souza, and R.~Skelton, ``Static output feedback
  controllers: Stability and convexity,'' \emph{IEEE Transactions on Automatic
  Control}, vol.~43, no.~1, pp. 120--125, 1998.

\bibitem{kwakernaak1972linear}
H.~Kwakernaak and R.~Sivan, \emph{Linear Optimal Control Systems}.\hskip 1em
  plus 0.5em minus 0.4em\relax Wiley-Interscience New York, 1972.

\bibitem{horn2012matrix}
R.~A. Horn and C.~R. Johnson, \emph{Matrix analysis}.\hskip 1em plus 0.5em
  minus 0.4em\relax Cambridge university press, 2012.

\bibitem{ciarlet1989introduction}
P.~G. Ciarlet, P.~G. Ciarlet, B.~Miara, and J.-M. Thomas, \emph{Introduction to
  numerical linear algebra and optimisation}.\hskip 1em plus 0.5em minus
  0.4em\relax Cambridge University Press, 1989.

\bibitem{boyd1994linear}
S.~Boyd, L.~El~Ghaoui, E.~Feron, and V.~Balakrishnan, \emph{Linear Matrix
  Inequalities in System and Control Theory}.\hskip 1em plus 0.5em minus
  0.4em\relax SIAM, 1994, vol.~15.

\end{thebibliography}

\appendices

\section{Proof of~\cref{lemma:monotone-norm}}\label{appendix:1}
We will use the following two lemmas:
\begin{lemma}[{\cite[Chap.~2]{boyd1994linear}}]\label{lemma:4}
For any $A \in {\mathbb S}^{n+m}$, $\lambda_{\max}(A)$ can be characterized by the following optimization
\begin{align*}
&\lambda ^* : = \arginf_{\lambda  \in {\mathbb R}} \lambda\\
&{\rm{subject}}\,\,{\rm{to}}\quad A \prec \lambda I.
\end{align*}
\end{lemma}
\begin{lemma}[{\cite[Corollary~7.7.13]{horn2012matrix}}]\label{lemma:5}
Let $A,B \in {\mathbb S}_+^{n+m}$. The following statements are equivalent:
\begin{enumerate}
\item $A\succeq B$

\item $\left[ {\begin{array}{*{20}c}
   A & B  \\
   B & A  \\
\end{array}} \right] \succeq 0$
\end{enumerate}
\end{lemma}

Consider any $A,B \in {\mathbb S}_+^{n+m}$ such that $A\succeq B$. Then, $\left\| A \right\|_P  \le \left\| B \right\|_P$ is equivalent to $\lambda _{\max } (A^T PA) \le \lambda _{\max } (B^T PB)$. Using~\cref{lemma:4}, the inequality can be cast as $\lambda _A^*  \le \lambda _B^*$, where
\begin{align*}
&\lambda_A ^* : = \arginf_{\lambda  \in {\mathbb R}} \lambda\\
&{\rm{subject}}\,\,{\rm{to}}\quad A^TPA \prec \lambda I
\end{align*}
and
\begin{align*}
&\lambda_B ^* : = \arginf_{\lambda  \in {\mathbb R}} \lambda\\
&{\rm{subject}}\,\,{\rm{to}}\quad B^TPB \prec \lambda I.
\end{align*}
Using the Schur complement~\cite[Chap.~2]{boyd1994linear}, the linear matrix inequality constraints can be equivalently written as
\begin{align}
\left[ {\begin{array}{*{20}c}
   { - \lambda I} & 0  \\
   0 & {-P^{ - 1} }  \\
\end{array}} \right] + \left[ {\begin{array}{*{20}c}
   0 & A  \\
   A & 0  \\
\end{array}} \right] \prec 0\label{eq:6}
\end{align}
and
\begin{align}
\left[ {\begin{array}{*{20}c}
   { - \lambda I} & 0  \\
   0 & {-P^{ - 1} }  \\
\end{array}} \right] + \left[ {\begin{array}{*{20}c}
   0 & B  \\
   B & 0  \\
\end{array}} \right] \prec 0.\label{eq:7}
\end{align}
On the other hand, by~\cref{lemma:5}, $A \preceq B \Leftrightarrow A - B \preceq 0$ implies
\[
\left[ {\begin{array}{*{20}c}
   0 & {B - A}  \\
   {B - A} & 0  \\
\end{array}} \right] \preceq 0 \Leftrightarrow \left[ {\begin{array}{*{20}c}
   0 & B  \\
   B & 0  \\
\end{array}} \right] \preceq \left[ {\begin{array}{*{20}c}
   0 & A  \\
   A & 0  \\
\end{array}} \right]
\]
The above inequality implies that if~\eqref{eq:7} is satisfied, then so is~\eqref{eq:6}. Therefore, it is easy to prove that $\lambda _A^*  \le \lambda _B^*$ holds. This completes the proof.

\section{Proof of~\cref{lemma:1}}\label{appendix:2}
Define
\[
P = \alpha \sum\limits_{k = 0}^\infty  {\left( {\frac{1}{{\rho (A(F)) + \varepsilon }}} \right)^{2k} (A(F)^T )^k A(F)^k }
\]
where $\alpha >0$ is any number. We will prove that it is the Lyapunov matrix stated in~\cref{lemma:1}. From the above definition, it is clear that $P$ satisfies
\[
A(F)^T P A(F) + \alpha I = (\rho (A(F)) + \varepsilon )^2 P
\]
It remains to prove the existence of such a matrix $P$, and the second statement.
We have
\begin{align*}
\lambda _{\max } (P) \le& \alpha \sum\limits_{k = 0}^\infty  {\left( {\frac{1}{{\rho (A(F) ) + \varepsilon }}} \right)^{2k} \lambda _{\max } ((A(F)^T )^k A(F)^k )}\\
=& \alpha \sum\limits_{k = 0}^\infty  {\left( {\frac{1}{{\rho (A(F) ) + \varepsilon }}} \right)^{2k} \left\| {A(F)^k } \right\|_2^2 }\\
\leq & \alpha \sum\limits_{k = 0}^\infty  {\left( {\frac{{\left\| {A(F)^k } \right\|_2^{1/k} }}{{\rho (A(F) ) + \varepsilon }}} \right)^{2k} }\\
=& \alpha \sum\limits_{k = 0}^{N - 1} {\left( {\frac{{\left\| {A(F)^k } \right\|_2^{1/k} }}{{\rho (A(F) ) + \varepsilon }}} \right)^{2k} } \\
+& \alpha \sum\limits_{k = N}^\infty  {\left( {\frac{{\left\| {A(F)^k } \right\|_2^{1/k} }}{{\rho (A(F) ) + \varepsilon }}} \right)^{2k} }
\end{align*}
From the Gelfand's formula~\cite[Corollary~5.6.14,~pp.~349]{horn2012matrix}, there exists a finite $N>0$ such that
\[
\left\| {A(F)^k } \right\|_2^{1/k}  \le \rho (A(F)) + 0.5\varepsilon ,\quad \forall k \ge N
\]
Combining the two inequalities leads to
\begin{align*}
&\lambda _{\max } (P)\\
\le& \alpha \sum\limits_{k = 0}^{N - 1} {\left( {\frac{{\left\| {A_{F}^k } \right\|_2^{1/k} }}{{\rho (A(F) ) + \varepsilon }}} \right)^{2k} }\\
&+ \alpha \sum\limits_{k = N}^\infty  {\left( {\frac{{\rho (A_{F} ) + 0.5\varepsilon }}{{\rho (A(F) ) + \varepsilon }}} \right)^{2k} }\\
=& \alpha \sum\limits_{k = 0}^{N - 1} {\left( {\frac{{\left\| {A(F)^k } \right\|_2^{1/k} }}{{\rho (A(F)  + \varepsilon }}} \right)^{2k} }\\
& + \alpha \left( {\frac{{\rho (A(F) ) + 0.5\varepsilon }}{{\rho (A(F) ) + \varepsilon }}} \right)^{2N} \sum\limits_{k = 0}^\infty  {\left( {\frac{{\rho (A(F) ) + 0.5\varepsilon }}{{\rho (A(F^*) ) + \varepsilon }}} \right)^{2k} }\\
=& \alpha \sum\limits_{k = 0}^{N - 1} {\left( {\frac{{\left\| {A(F)^k } \right\|_2^{1/k} }}{{\rho (A(F) ) + \varepsilon }}} \right)^{2k} }\\
&+ \alpha \left( {\frac{{\rho (A(F) ) + 0.5\varepsilon }}{{\rho (A(F)) + \varepsilon }}} \right)^{2N} \frac{1}{{1 - \left( {\frac{{\rho (A(F) ) + 0.5\varepsilon }}{{\rho (A(F) ) + \varepsilon }}} \right)^2 }}\\
=&\alpha \sum\limits_{k = 0}^{N - 1} {\left( {\frac{{\left\| {A(F)^k } \right\|_2^{1/k} }}{{\rho (A(F) ) + \varepsilon }}} \right)^{2k} }\\
&  + \alpha \left( {\frac{{\rho (A(F) ) + 0.5\varepsilon }}{{\rho (A(F) ) + \varepsilon }}} \right)^{2N} \frac{{\rho (A(F) ) + \varepsilon }}{{0.5\varepsilon }}
\end{align*}
Therefore, $\lambda _{\max } (P)$ is finite. By setting $\alpha$ to be the inverse of the upper bound, we have the desired result.

\section{Proof of~\cref{thm:7}}\label{appendix:3}
Let $P_k  = {\cal T}^k (P)$ and $F_k  =  - P_{k,22}^{ - 1} P_{k,12}^T$. First of all, we have
\begin{align*}
&{\cal T}(P) - {\cal T}(P^* )\\
=& \gamma A( - P_{22}^{ - 1} P_{12}^T )^T PA( - P_{22}^{ - 1} P_{12}^T ) - \gamma A(F^* )^T P^* A(F^* )\\
\succeq& \gamma A( - P_{22}^{ - 1} P_{12}^T )^T P A( - P_{22}^{ - 1} P_{12}^T )\\
& - \gamma A( - P_{22}^{ - 1} P_{12}^T )^T P^* A( - P_{22}^{ - 1} P_{12}^T )\\
=& \gamma A( - P_{22}^{ - 1} P_{12}^T )^T (P - P^* )A( - P_{22}^{ - 1} P_{12}^T ),
\end{align*}
where the inequality is due to~\cref{lemma:key-result}. In general, using ${\cal T}(P^*) = P^*$, it follows that
\begin{align*}
&{\cal T}^{k + 1} (P) - {\cal T}^{k + 1} (P^* )\\
=& \gamma A(F_k )^T {\cal T}^k (P)A(F_k ) - \gamma A(F^* )^T P^* A(F^* )\\
\succeq & \gamma A(F_k )^T {\cal T}^k (P)A(F_k ) - \gamma A(F_k )^T P^* A(F_k )\\
 =& \gamma A(F_k )^T ({\cal T}^k (P) - {\cal T}^k (P^* ))A(F_k )
\end{align*}
for all $k\geq 0$, where the inequality is due to~\cref{lemma:key-result}. By recursively applying the last inequality, one gets
\[
{\cal T}^k (P) - {\cal T}^k (P^* ) \succeq \left( {\prod\limits_{i = 0}^{k - 1} {\gamma ^{1/2} A(F_i )} } \right)^T (P - P^* )\left( {\prod\limits_{i = 0}^{k - 1} {\gamma ^{1/2} A(F_i )} } \right)
\]
Since $P_k \to P^*$ as $k \to \infty$, $F_k \to F^*$ as well. Therefore, for any $\varepsilon_1>0$, there exists a sufficiently large $N_1>0$ such that
\[
\left\| {\gamma ^{1/2} A(F_k )) - \rho (\gamma ^{1/2} A(F^* ))} \right\|_2  \le \varepsilon_1 ,\quad \forall k \ge N
\]
holds, which implies
\[
\left\| {\gamma ^{1/2} A(F_k ))} \right\|_2  \le \left\| {\rho (\gamma ^{1/2} A(F^* ))} \right\|_2  + \varepsilon_1 ,\quad \forall k \ge N
\]
using the inverse triangular inequality. Now, by letting
\[
M = \left( {\prod\limits_{i = 0}^{N_1 - 1} {\gamma ^{1/2} A(F_i )} } \right)^T (P - P^* )\left( {\prod\limits_{i = 0}^{N - 1} {\gamma ^{1/2} A(F_i )} } \right)
\]
it follows that
\begin{align*}
&{\cal T}^{k + N_1} (P) - {\cal T}^{k + N_1} (P^* )\\
 \succeq& \left( {\prod\limits_{i = 1}^k {\gamma ^{1/2} A(F_i )} } \right)^T M\left( {\prod\limits_{i = 1}^k {\gamma ^{1/2} A(F_i )} } \right)\\
 \succeq&  - I\lambda _{\max } (M)\left\| {\prod\limits_{i = 1}^k {\gamma ^{1/2} A(F_i )} } \right\|_2^2\\
 \succeq&  - I\lambda _{\max } (M)\left( {\left\| {(\gamma ^{1/2} A(F^* ))^k } \right\|_2  + \varepsilon_1 } \right)^2\\
 =&  - I\lambda _{\max } (M)\left( {\left( {\left\| {(\gamma ^{1/2} A(F^* ))^k } \right\|_2^{1/k} } \right)^{k}  + \varepsilon_1 } \right)^2
\end{align*}
for all $k\geq 0$. From the Gelfand's formula~\cite[Corollary~5.6.14,~pp.~349]{horn2012matrix}, there exists a finite $N_2\geq 0$ such that
\[
\| {\gamma^{k/2} A(F^*)^k } \|_2^{1/k}  \le \rho (\gamma^{1/2} A(F^*)) + \varepsilon_2 ,\quad \forall k \ge N_2.
\]
for any $\varepsilon_2>0$. Plugging this bound into the previous inequality leads to
\begin{align*}
&{\cal T}^{k + N_1 } (P) - {\cal T}^{k + N_1 } (P^* )\\
\succeq&  - I\lambda _{\max } (M)\{ (\rho (\gamma ^{1/2} A(F^* )) + \varepsilon _2 )^k  + \varepsilon _1 \}^2
\end{align*}
for all $k\geq N_2$. Therefore, the desired result is obtained.
\end{document}